\documentclass[11pt]{article}
\usepackage{latexsym, amsmath, amsfonts, amscd, amssymb, verbatim}
\usepackage{color}
\usepackage{hyperref}
\usepackage[margin=1in]{geometry}  

\newtheorem{Theorem}{Theorem}[section]
\newtheorem{Proposition}{Proposition}[section]
\newtheorem{Remark}{Remark}[section]

\newtheorem{Lemma}{Lemma}[section]

\newtheorem{Definition}{Definition}[section]
\newtheorem{Example}{Example}[section]
\normalsize
\def\Limsup{\mathop{{\rm Lim}\,{\rm sup}}}

\makeatletter
\renewcommand\@biblabel[1]{#1.}
\makeatother

\begin{document}

\title{\bf On Second-order Conditions for  Quasiconvexity and Pseudoconvexity  of $\mathcal{C}^{1,1}$-smooth Functions}
\author{Pham  Duy Khanh\footnote{Department of Mathematics, HCMC University of Education, Ho Chi Minh, Vietnam and Center for Mathematical
		Modeling, Universidad de Chile, Santiago, Chile. E-mails:
		pdkhanh182@gmail.com; pdkhanh@dim.uchile.cl},\and Vo Thanh Phat\thanks{Department of Mathematics, HCMC University of Pedagogy,  Ho Chi Minh, Vietnam. E-mail:
		vtphat1996@gmail.com}}
\maketitle

\medskip
\begin{quote}
\noindent {\bf Abstract}  For a $\mathcal{C}^2$-smooth function on a finite-dimensional space, a necessary condition for its quasiconvexity is the positive semidefiniteness of its Hessian matrix on the subspace orthogonal to its gradient, whereas a sufficient condition for  its strict pseudoconvexity  is the positive definiteness of its Hessian matrix on the subspace orthogonal to its gradient. Our aim in this paper is to extend those conditions for  $\mathcal{C}^{1,1}$-smooth functions by using the  Fr\'echet and Mordukhovich second-order subdifferentials.


\end{quote}

\section{Introduction}
Since the notion of convexity does no longer suffice to many mathematical models used in
decision sciences, economics, management sciences, stochastics, applied mathematics and engineering, various generalizations of convex functions have been introduced in literature such as (strictly) quasiconvex and (strictly) pseudoconvex functions. Those functions  preserve one or more properties of convex functions and give rise to models which are more adaptable to real-world situations than convex models. Quasiconvex functions are characterized by the property that every level set is convex and pseudoconvex functions are such that the vanishing of the gradient ensures a global minimum. 

First-order characterizations for quasiconvexity and pseudoconvexity can be found in \cite{Arviel88,AlbertoLaura09,CrouzeixFerland82,Crouzeix98} for smooth functions and  \cite{Aussel94,Aussel98,BGJ13a,HoaKhanhTrinh18,Ivanov13,Luc93,Soleimani-damaneh07,Yang05} for nonsmooth ones. 
 The well-known second-order necessary condition for the quasiconvexity of $\mathcal{C}^2$-smooth functions (see for instance \cite{Avriel72,AlbertoLaura09,Crouzeix,Crouzeix98}) states that the Hessian matrix of a quasiconvex function is positive-semidefinite on the subspace orthogonal to its gradient.  Furthermore, if the Hessian matrix of a $\mathcal{C}^2$-smooth function is positive definite on the subspace orthogonal to its gradient then the given function is strictly pseudoconvex \cite{Crouzeix98}.  Relaxing the $\mathcal{C}^2$-smoothness property, several authors have characterized the quasiconvexity and pseudoconvexity by using various kinds of generalized second-order derivatives.  Using Taylor's formula and an estimation formula of generalized Hessian, Luc \cite{Luc95} established the necessary and sufficient conditions for quasiconvexity of $\mathcal{C}^{1,1}$-smooth functions. In \cite{GI03} Ginchev and Ivanov  introduced the concept of second-order upper Dini-directional derivatives and used it to characterize the pseudoconvexity of radially upper semicontinuous functions.
 Recently, by using the theory of viscosity solutions of partial differential equations, Barron, Goebel, 
 and Jensen \cite{BGJ13b} obtained some necessary conditions and sufficient ones for the quasiconvexity of upper semicontinuous functions.
 
Although the Fr\'echet and the Mordukhovich second-order subdifferentials 
have a significant role in variational analysis and its applications \cite{BorisNamYen06,Mordukhovich06,Rockafellar98}, it seems to us that their use in characterizing quasiconvexity and pseudoconvexity of nonsmooth functions has not been studied, so far.
This motivates us to find out some necessary conditions  and sufficient ones for  the quasiconvexity and  pseudoconvexity of $\mathcal{C}^{1,1}$-smooth functions by using the  Fr\'echet and Mordukhovich second-order subdifferentials. For the necessity, we prove that the Fr\'echet second-order subdifferential of a pseudoconvex function is positive semidefinite on the subspace orthogonal to its gradient  while the Mordukhovich second-order subdifferential of a quasiconvex function  is only positive semidefinite along its some selection. 
For the sufficiency, we propose two conditions guaranteeing the strict pseudoconvexity.
The first one is the positive definiteness of the Mordukhovich second-order subdifferential of a given function on the subspace orthogonal to its gradient. The second one claims that the Fr\'echet second-order subdifferential of a given function has some selection  which is positive on the subspace orthogonal to its gradient. Moreover, a second-order sufficient condition for the strict quasiconvexity is also established by using Fr\'echet second-order subdifferentials. 
The obtained results are analyzed and illustrated by suitable examples.

The paper is organized as follows. Some background material from variational analysis and generalized convexity are recalled in  Section 2.  Section 3 presents some second-order conditions for quasiconvexity and pseudoconvexity of  $\mathcal{C}^{1,1}$-smooth functions. Sufficient conditions are given in Section~4. Conclusions and further investigations are discussed in the last section.

\section{Preliminaries}
 We begin with some notions from \cite{Mordukhovich06} which will be needed in the sequel. Let $F$ be a set-valued mapping  between Euclidean spaces $\mathbb{R}^n$ and $\mathbb{R}^m$. As usual, the effective domain and the graph of $F$ are given, respectively, by
 $$
 \text{dom}F:=\{x\in \mathbb{R}^n|F(x)\ne\emptyset\}\quad\text{and}\quad\text{gph}F=\{(x,y)\in\mathbb{R}^n\times\mathbb{R}^m|y\in F(x)\}.
 $$
The \textit{sequential Painlev\'e-Kuratowski outer limit}  of $F$ as $x\rightarrow\bar{x}$ is defined as
\begin{equation}\label{OuterLimit}
\begin{array}{rl}
\underset{x\rightarrow\bar{x}}{\Limsup}\;F(x):
=\{y\in \mathbb{R}^n\, |\,  &\exists \mbox{ sequences } x_k\to \bar x,\ y_k\rightarrow y,\\
&\mbox {with } y_k\in F(x_k)\ \mbox{for all}\  k=1,2,\dots
\}.
\end{array}
\end{equation}
Let us consider an extended-real-valued function $\varphi:\mathbb{R}^n\rightarrow\overline{\mathbb{R}}:=(-\infty,\infty]$. We always assume that $\varphi$  is proper and lower semicontinuous. The \textit{Fr\'echet subdifferential} of $\varphi$ at $\bar{x}\in\text{dom}\varphi:=\{x\in\mathbb{R}^n:\varphi(x)<\infty\}$ 
(known as the presubdifferential and as the regular or viscosity subdifferential) is
\begin{equation}\label{FrechetSubdifferential}
\widehat \partial\varphi\left( {\bar{x} } \right): = \left\{ {x^* \in {\mathbb{R}^n}\, |\,\mathop {\lim \inf }\limits_{x \to \overline x } \frac{{\varphi (x) - \varphi \left( {\bar{x} } \right) - \left\langle {{x^*},x - \bar{x}} \right\rangle }}{{\left\| {x - \bar{x} } \right\|}} \ge 0} \right\}.
\end{equation}
Then the \textit{Mordukhovich subdifferential} of $\varphi$ at $\bar{x}$ (known also the general/basic or limiting subdifferential) is defined via the outer limit \eqref{OuterLimit} 
\begin{equation}\label{MordukhovichSubdifferential}
\partial \varphi(\bar{x}):= \underset{x \overset{\varphi}{\to} \bar{x}} {\Limsup} \; \widehat{\partial}\varphi(x),
\end{equation}
where $x \overset{\varphi}{\to} \bar{x}$ signifies that $x\rightarrow\bar{x}$ with $\varphi(x)\rightarrow \varphi(\bar{x})$. Observe that both Fr\'echet and Mordukhovich subdifferentials reduce to the classical Fr\'echet derivative for continuously differentiable functions.

Given a set $\Omega\subset\mathbb{R}^n$ with its indicator function $\delta_\Omega(x)$ equal to $0$ for $x\in\Omega$ and to $\infty$ otherwise, the Fr\'echet and the Mordukhovich \textit{normal cones} to $\Omega$ at $\bar{x}\in\Omega$ are defined, respectively, via the corresponding subdifferentials \eqref{FrechetSubdifferential} and \eqref{MordukhovichSubdifferential} by
\begin{equation} \label{NormalCones}
\widehat{N}(\bar{x};\Omega):=\widehat{\partial}\delta_\Omega(\bar{x})
\quad\text{and}\quad
N(\bar{x};\Omega):=\partial\delta_\Omega(\bar{x}).
\end{equation}
The Fr\'echet and Mordukhovich coderivatives of $F$ at $(\bar{x},\bar{y})\in\text{gph}F$ are defined, respectively, via corresponding normal cones \eqref{NormalCones} by
$$
D^*F(\overline x ,\overline y )(y^*): = \{ x^* \in {\mathbb{R}^n}:({x^*}, - {y^*}) \in {N}((\bar{x},\bar{y}), \text{gph}F) \},
$$
$$
\widehat D^*F(\overline x ,\overline y )(y^*): = \{ x^* \in {\mathbb{R}^n}:({x^*}, - {y^*}) \in \widehat{N}((\bar{x},\bar{y}), \text{gph}F) \}.
$$
We omit $\bar{y} = f(\bar{x})$ in the above coderivative notions if $F:= f:\mathbb{R}^n \to \mathbb{R}^m$ is single-valued.
\begin{Definition}\rm 
Let $\varphi: \mathbb{R}^n \to \overline{\mathbb{R}}$  be a function with a finite value at $\bar{x}$. 
\begin{itemize}
	\item[(i)] For any $\bar{y} \in \partial \varphi(\bar{x})$, the map $\partial^2 \varphi(\bar{x},\bar{y}): \mathbb{R}^n \rightrightarrows \mathbb{R}^n$ with the values
	$$
	\partial^2 \varphi(\bar{x},\bar{y})(u)= (D^*\partial\varphi)(\bar{x},\bar{y})(u) \quad (u \in \mathbb{R}^n)
	$$
	is said to be the \textit{Mordukhovich second-order subdifferential} of $\varphi$ at $\bar{x}$ relative to $\bar{y}$. 
	\item[(ii)] For any $\bar{y} \in \widehat{\partial} \varphi(\bar{x})$, the map $\widehat{\partial}^2 \varphi(\bar{x},\bar{y}): \mathbb{R}^n \rightrightarrows \mathbb{R}^n$ with the values
	$$
	\widehat{\partial}^2 \varphi(\bar{x},\bar{y})(u)= (\widehat{D}^*\widehat{\partial}\varphi)(\bar{x},\bar{y})(u) \quad (u \in \mathbb{R}^n)
	$$
	is said to be the \textit{Fr\'echet second-order subdifferential } of $\varphi$ at $\bar{x}$ relative to $\bar{y}$. 
\end{itemize}
We omit $\bar{y} = \nabla \varphi(\bar{x})$ in the above second-order subdifferentials if $\varphi\in\mathcal{C}^1$ around $\bar{x}$, i.e., continuously Fr\'echet differentiable in a neighborhood of $\bar{x}$.
\end{Definition}
In general, the Fr\'echet second-order subdifferential and the Mordukhovich one are incomparable. However, if $\varphi\in\mathcal{C}^1$ around $\bar{x}$, then
\begin{equation}\label{Frechet_Mordukhovich}
\widehat{\partial}^2 \varphi(\bar{x})(u)\subset\partial^2 \varphi(\bar{x})(u), \quad \forall u \in \mathbb{R}^n.
\end{equation}
If $\varphi\in\mathcal{C}^{1,1}$ around $\bar{x}$, i.e., Fr\'echet  differentiable around $\bar{x}$ with the gradient $\nabla\varphi$ being locally Lipschitzian around $\bar{x}$ then the calculation of second-order subdifferentials can be essentially simplified due to the following scalarization formulas (see \cite[Proposition 3.5]{BorisNamYen06} and \cite[Proposition 1.120]{Mordukhovich06})
\begin{equation}\label{ScalarizationFormular}
\widehat{\partial}^2 \varphi(\bar{x})(u)=\widehat\partial\langle u,\nabla\varphi\rangle(\bar{x}), \quad
{\partial}^2 \varphi(\bar{x})(u)=\partial\langle u,\nabla\varphi\rangle(\bar{x}).
\end{equation}
In this case, Mordukhovich second-order subdifferentials are nonempty
 \cite[Corollary~2.25]{Mordukhovich06} while 
Fr\'echet ones may be empty. If $\varphi\in\mathcal{C}^2$ around $\bar{x}$, i.e., $\varphi$ is twice continuously Fr\'echet differentiable in a
neighborhood of $\bar{x}$, then
\begin{equation}\label{C2_Case}
\partial^2 \varphi(\bar{x})(u)=\widehat{\partial}^2 \varphi(\bar{x})(u)= \{\nabla^2 \varphi(\bar{x})u \}, \quad \forall u \in \mathbb{R}^n.
\end{equation}

\medskip
Let us recall some well-known notions of generalized convexity.
\begin{Definition}\rm  \text{} 
\begin{itemize}
	\item[(a)] A function $\varphi: \mathbb{R}^n \to \mathbb{R}$ is said to be \textit{quasiconvex}  if
	$$
	\varphi((1-\lambda)x + \lambda y) \leq \max\{\varphi(x),\varphi(y)\}
	$$
	for every $x, y \in \mathbb{R}^n$  and for every $\lambda \in [0,1]$.
	\item[(b)] A function $\varphi: \mathbb{R}^n \to \mathbb{R}$ is said to be \textit{strictly quasiconvex}  if
	$$
	\varphi((1-\lambda)x + \lambda y) < \max\{\varphi(x),\varphi(y)\}
	$$
	for every $x, y \in \mathbb{R}^n, x\ne y$  and for every $\lambda \in (0,1)$.
	\item[(c)] A differentiable function $\varphi: \mathbb{R}^n \to \mathbb{R}$ is called \textit{pseudoconvex}  if
	$$
	x,y \in \mathbb{R}^n, \varphi(x)>\varphi(y) \Longrightarrow \langle \nabla \varphi(x), y - x \rangle <0.
	$$
	\item[(d)] A differentiable function $\varphi: \mathbb{R}^n \to \mathbb{R}$ is called \textit{strictly pseudoconvex}  if
	$$
	x,y \in \mathbb{R}^n,\; x\ne y,\;  \varphi(x)\geq\varphi(y) \Longrightarrow \langle \nabla \varphi(x), y - x \rangle <0.
	$$
\end{itemize}
\end{Definition}
It follows, immediately, from the given definitions, that a strictly quasiconvex (pseudoconvex)
function is quasiconvex (pseudoconvex). For differentiable functions,  (strict) pseudoconvexity implies (strict) quasiconvexity. 
The next theorem shall point out that within the class of (strictly) quasiconvex functions, (strict) pseudoconvexity may be specified by means of its behaviour
at a critical point.

\begin{Theorem}{\rm \cite[Theorem 3.2.9]{AlbertoLaura09}}\label{QuasiPseudo}
	Let $\varphi:\mathbb{R}^n\rightarrow\mathbb{R}$ be a continuous differentiable function. Then, $\varphi$ is (strictly) pseudoconvex if
	and only if the following conditions hold:
	\begin{itemize}
		\item[{\rm (i)}] $\varphi$ is quasiconvex;
		\item[{\rm (ii)}] If $\nabla\varphi(x)=0$ then $x$ is a (strict) local minimum for $\varphi$. 
	\end{itemize}
\end{Theorem}

\medskip
Finally, we consider a lemma which will be used in the subsequent.
\begin{Lemma}\label{Quasiconvex}
	Let $\varphi:\mathbb{R}^n\rightarrow\mathbb{R}$ be a differentiable function. If $\varphi$ is not strictly quasiconvex, then there exist $x_1, x_2\in\mathbb{R}^n, x_1\ne x_2$  and $t_0\in (0,1)$ such that $\langle\nabla\varphi(x_1+t_0(x_2-x_1)),x_2-x_1\rangle=0$ and
	\begin{equation}\label{LocalMax}
	\varphi(x_1+t(x_2-x_1))\leq\varphi(x_1+t_0(x_2-x_1)), \quad \forall t\in[0,1].
	\end{equation}	
\end{Lemma}
\begin{proof} Since $\varphi$ is not strictly quasiconvex, there exist $x_1,x_2\in\mathbb{R}^n, \alpha\in(0,1)$ such that $x_1\ne x_2$ and
	\begin{equation}\label{NotQuasiconvex}
	\varphi((1-\alpha)x_1+\alpha x_2)\geq\max\{\varphi(x_1),\varphi(x_2)\}.
	\end{equation}
	Consider the function $f:\mathbb{R}\rightarrow\mathbb{R}$ given by
	$$
	f(t)=\varphi(x_1+t(x_2-x_1)),\quad \forall t\in\mathbb{R}.
	$$
	Then, thanks to the Weierstrass theorem and \eqref{NotQuasiconvex},
we can find a number $t_0\in (0,1)$ for which the function $f$
admits a maximum on the interval $[0, 1]$. Hence, \eqref{LocalMax} is satisfied and
	by the Fermat rule we have
	$$
	0=\nabla f(t_0)=\langle\nabla\varphi(x_1+t_0(x_2-x_1)),x_2-x_1\rangle.
	$$
	$\hfill\Box$
\end{proof}

\section{Necessary Conditions}
Let us recall the well-known second-order necessary condition for quasiconvexity of $\mathcal{C}^2$-smooth functions.
\begin{Theorem}\label{NecessaryQCC2} {\rm (see \cite[Lemma 6.2]{Avriel72} or \cite[Theorem 3.4.2]{AlbertoLaura09})}
Let $\varphi:\mathbb{R}^n \to \mathbb{R}$ be a $\mathcal{C}^2$-smooth function. If $\varphi$ is quasiconvex, then 
\begin{equation}\label{Eq_NecessaryQCC2}
x, u \in \mathbb{R}^n, \langle \nabla \varphi(x),u \rangle = 0  \Longrightarrow \langle \nabla^2\varphi(x)u,u \rangle \geq 0.
\end{equation} 
\end{Theorem}

By using the mean value inequality in terms of Mordukhovich subdifferential  for Lipschitzian functions \cite[Corollary~3.51 ]{Mordukhovich06} we extend the above result to $\mathcal{C}^{1,1}$-smooth functions.  
\begin{Proposition}\label{MeanValue}
	Let $\varphi:\mathbb{R}^n\rightarrow\mathbb{R}$ be a Lipschitz continuous function on open set containing $[a,b]$. Then one has
	$$
	\langle x^*,b-a\rangle\geq \varphi(b)-\varphi(a)\quad\text{for some}\quad x^*\in\partial\varphi(c),\quad c\in [a,b).
	$$
\end{Proposition}
\begin{Theorem}\label{NecessaryQC}  Let $\varphi:\mathbb{R}^n\rightarrow\mathbb{R}$ be a $\mathcal{C}^{1,1}$-smooth function. If $\varphi$ is quasiconvex then 
	\begin{equation}\label{Eq_NecessaryQC}
		x, u \in \mathbb{R}^n, \langle \nabla \varphi(x), u \rangle = 0 \Longrightarrow \langle z, u \rangle \geq 0 \; \text{for some} \; z \in {\partial}^2 \varphi(x)(u).
	\end{equation}
\end{Theorem}
\begin{proof}
	Let $x,u\in\mathbb{R}^n$ be such that $\langle\nabla\varphi(x),u\rangle=0$. If $u=0$ then 
	$\langle z,u\rangle= 0$ for all  $z\in{\partial}^2 \varphi(x)(u)$.
	Otherwise, consider the function $f:\mathbb{R}^n\rightarrow\mathbb{R}$ given by 
	$$
	f(y):=\langle\nabla\varphi(y),u\rangle\quad \forall y\in\mathbb{R}^n.
	$$ 
	Then, $f(x)=0$ and $f$ is locally Lipschitz continuous on $\mathbb{R}^n$ by the $\mathcal{C}^{1,1}$-smoothness of $\varphi$. Moreover, $\partial f$ is locally bounded (see \cite[Corollary~1.81]{Mordukhovich06} or \cite[Theorem~9.13]{Rockafellar98}), robust \cite[Proposition~8.7]{Rockafellar98} on $\mathbb{R}^n$ and for every $y\in\mathbb{R}^n$
	$$
	\partial f(y)=\partial\langle u,\nabla\varphi\rangle(y)=\partial^2\varphi(y)(u).
	$$ 
	For the sequences $x_k:=x+(1/k)u,\; x^\prime_k:=x-(1/k)u\; (k\in\mathbb{N})$, one has $x_k\rightarrow x, x^\prime_k\rightarrow x$ and, in view of Proposition~\ref{MeanValue}, there exist $\theta_k \in [0,1/k)$, $\theta^\prime_k\in (0,1/k]$ and $z_k\in\partial f(x+\theta_ku), z^\prime_k\in\partial f(x-\theta_ku)$ such that
	$$
	\langle z_k,(1/k)u\rangle=\langle z_k,x_k-x\rangle\geq f(x_k)-f(x)=f(x_k),
	$$
	$$
	\langle z^\prime_k,(1/k)u\rangle=\langle z^\prime_k,x-x^\prime_k\rangle\geq f(x)-f(x^\prime_k)=-f(x^\prime_k).
	$$
	By the quasiconvexity of $\varphi$, it follows from \cite[Proposition~1]{Crouzeix98} that
	\begin{eqnarray*}
		0&\geq& \min\{\langle\nabla\varphi(x_k),x^\prime_k-x_k\rangle, \langle\nabla\varphi(x^\prime_k),x_k-x^\prime_k\rangle\}\\
		&=& \min\{(-2/k)f(x_k), (2/k)f(x^\prime_k)\}\\
		&\geq& \min\{(-2/k^2)\langle z_k,u\rangle, (-2/k^2)\langle z^\prime_k,u\rangle\}.
	\end{eqnarray*}
	Therefore, $\max\{\langle z_k,u\rangle,\langle z^\prime_k,u\rangle\}\geq 0$ for all $k\in\mathbb{N}$. Since $\partial f$ is locally bounded at $x$, the sequences $(z_k), (z^\prime_k)$ are bounded. Without loss of generality, we can assume that $z_k\rightarrow z$ and $z^\prime_k\rightarrow z^\prime$. It follows that $\max\{\langle z,u\rangle,\langle z^\prime,u\rangle\}\geq 0$ and by the robustness of $\partial f$ we have $z,z^\prime\in\partial f(x)=\partial^2\varphi(x)(u)$. The proof is complete.
	$\hfill\Box$
\end{proof}

\medskip
The established necessary condition says that, for a quasiconvex function, the Mordukhovich second-order
subdifferential at one point is positive semidefinite along some its selection on the subspace orthogonal to its gradient at this point.  The following example shows that the positive semidefiniteness cannot be extended to the whole mentioned subspace even for pseudoconvex function.

\begin{Example}\rm \cite[Remark 3.1]{HuyTuyen16}
	Let $\varphi: \mathbb{R} \to \mathbb{R}$ be defined by
	$$
	\varphi(x):= \int_0^{|x|} {\phi (t)dt},
	$$
	where
	$$
	\phi(t) = \begin{cases}
	2t^2 + t^2 \sin \left(\frac{1}{t} \right)& \text{if}\qquad t > 0,\\
	0 & \text{if}\qquad t=0.\\
	\end{cases}
	$$
	Observe that $\varphi$ is a pseudoconvex $\mathcal{C}^{1,1}$-smooth function. Indeed, for every $x\in\mathbb{R}$, we have 
	$$
	\nabla \varphi(x) = \begin{cases}
	\phi(x)& \text{if}\qquad x \geq 0,\\
	-\phi(-x) & \text{if}\qquad x< 0.
	\end{cases}
	$$
	Hence, $\nabla\varphi$ is locally Lipschitz and so it is $\mathcal{C}^{1,1}$-smooth. Moreover,  $\nabla \varphi(x) = 0$ if and only if $x = 0$ and $0$ is a local minimum of $\varphi$. It follows from \cite[Theorem 3.2.7]{AlbertoLaura09} that $\varphi$ is a pseudoconvex function. Clearly, one has $\partial^2 \varphi(0)(u) = [-|u|,|u|]$ for each $u \in \mathbb{R}$. Thus, with $u\ne 0$, there exists $z^* \in \partial^2 \varphi(0)(u) $  such that $\langle z^*, u \rangle < 0$. 
\end{Example}

Although the pseudoconvexity does not imply the positive semidefiniteness of the second-order Mordukhovich subdifferential, it guarantees the positive semidefiniteness of the second-order Fr\'echet subdifferential.

\begin{Theorem}\label{NecessaryPC} Let $\varphi:\mathbb{R}^n\rightarrow\mathbb{R}$ be a $\mathcal{C}^{1,1}$-smooth function. If $\varphi$ is pseudoconvex then 
\begin{equation}\label{NormalCone_PC}
x, u \in \mathbb{R}^n, \langle \nabla \varphi(x), u \rangle = 0 \Longrightarrow \langle z, u \rangle \geq 0 \; \text{for all} \; z \in \widehat{\partial}^2 \varphi(x)(u).
\end{equation}
\end{Theorem}
\begin{proof}
Suppose to the contrary that there exist	
$x,u\in\mathbb{R}^n$ and $z\in \widehat{\partial}^2 \varphi(x)(u)$  such that $\langle \nabla \varphi(x), u \rangle = 0$ and $\langle z,u\rangle<0$. By \eqref{ScalarizationFormular}, we have
$z\in\widehat{\partial}\langle u,\nabla\varphi\rangle(x)$ and so
\begin{equation}\label{NormalCone}
\begin{split}
0
&\leq \liminf_{y\rightarrow x}\frac{\langle u,\nabla\varphi(y)\rangle-\langle u,\nabla\varphi(x)\rangle-\langle z,y-x\rangle}{\|y-x\|}\\
&=\liminf_{y\rightarrow x}\frac{\langle u,\nabla\varphi(y)\rangle-\langle z,y-x\rangle}{\|y-x\|}.
\end{split}
\end{equation}
For the sequence $x_k:=x-(1/k)u\; (k\in\mathbb{N})$, one has $x_k\rightarrow x$ and 
$$
 \langle \nabla \varphi(x), x_k - x \rangle = \langle \nabla \varphi(x), -(1/k)u \rangle =0.
$$
The pseudoconvexity of $\varphi$ implies that $\varphi(x_k)\geq\varphi(x)$ and by the classical mean value theorem there exists $\theta_k\in (0,1/k)$ such that 
$$
0\leq\varphi(x_k)-\varphi(x)=\langle\nabla\varphi(x-\theta_ku),x_k-x\rangle=\langle\nabla\varphi(x-\theta_ku),(-1/k)u\rangle.
$$
For the sequence $y_k:=x-\theta_ku\; (k\in\mathbb{N})$, one has $y_k\rightarrow x$ and $\langle\nabla\varphi(y_k),u\rangle\leq 0$. 
Therefore, by \eqref{NormalCone} we have
\begin{equation*}
\begin{split}
0
&\leq\liminf_{k\rightarrow\infty}\frac{\langle u,\nabla\varphi(y_k)\rangle-\langle z,y_k-x\rangle}{\|y_k-x\|}\\
& \leq \liminf_{k\rightarrow\infty}\frac{\langle z,\theta_ku\rangle}{\|\theta_ku\|}\\
& =\frac{\langle z,u\rangle}{\|u\|}
\end{split}
\end{equation*}
which is contradict to $\langle z,u\rangle<0$.
$\hfill\Box$
\end{proof}

\medskip
The next example shows that \eqref{NormalCone_PC} is  violated if the pseudoconvexity is relaxed to quasiconvexity.
\begin{Example}\label{Signfunction}\rm 
Let $\varphi: \mathbb{R} \to \mathbb{R}$ be given by 
$$
\varphi(x) := \frac{1}{2} x^2 \text{sign}x,  \quad \forall x \in \mathbb{R}.
$$	
Observe that $\varphi$ is a quasiconvex $\mathcal{C}^{1,1}$-smooth function and $\nabla \varphi(x) = |x|$  for every $x \in \mathbb{R}$. Moreover, we have 
$$
\widehat{\partial}^2 \varphi(0)(u) = \begin{cases}
	[-u,u]& \text{if}\qquad u \geq 0,\\
	\emptyset & \text{if}\qquad u< 0.
\end{cases}
$$
Observe that for $z=-1, u=1$, we have $z \in \widehat{\partial}^2 \varphi(0)(u)$  and $\langle z,u \rangle < 0$.
\end{Example}

\section{Sufficient Conditions}
A second-order sufficient condition for the strict pseudoconvexity in the $\mathcal{C}^2$-smoothness case is recalled in the following theorem.
\begin{Theorem}\label{SufficientSPC} {\rm \cite[Proposition~4]{Crouzeix98}}
Let $\varphi: \mathbb{R}^n \to \mathbb{R}$ be a $\mathcal{C}^2$-smooth function satisfying 
\begin{equation}\label{Eq_SufficientSPC}
x\in \mathbb{R}^n, u \in \mathbb{R}^n \setminus \{0\}, \langle \nabla \varphi(x),u \rangle =0 \Longrightarrow \langle \nabla^2\varphi(x)u,u \rangle >0.
\end{equation}  
Then, $\varphi$ is a strictly pseudoconvex function. 
\end{Theorem} 

Our aim in this section is to establish some similar versions of Theorem~\ref{SufficientSPC} in the $\mathcal{C}^{1,1}$-smoothness case by using the Fr\'echet and Mordukhovich second-order subdifferentials.  The first version is the replacement of the Hessian matrices  in \eqref{Eq_SufficientSPC} by the Mordukhovich second-order subdifferentials.
Our proof is based on Theorem~\ref{QuasiPseudo} and the following sufficient optimality condition for $\mathcal{C}^{1,1}$-smooth functions.

\begin{Proposition} {\rm \cite[Corollary 4.8]{ChieuLeeYen17}} \label{SubfficientLocal}
	Suppose that $\varphi:\mathbb{R}^n\rightarrow\mathbb{R}$ is a $\mathcal{C}^{1,1}$-smooth function and $x\in\mathbb{R}^n$. If $\nabla\varphi(x)=0$ and 
	$$
	\langle z,u\rangle>0\;  \text{for all}\;   z\in \partial^2\varphi(x)(u), u\in\mathbb{R}^n
	$$ 
	then $x$ is a strict local minimizer of $\varphi$. 
\end{Proposition}
\begin{Theorem}\label{SufficientPCC11}   Let $\varphi:\mathbb{R}^n\rightarrow\mathbb{R}$ be a $\mathcal{C}^{1,1}$-smooth function satisfying
\begin{equation}\label{SufficientPC}
x \in \mathbb{R}^n, u \in \mathbb{R}^n \setminus \{0\}, \langle \nabla \varphi(x), u \rangle=0 \Longrightarrow \langle z, u \rangle >0 \; \text{for all} \; z \in \partial^2 \varphi(x)(u).
\end{equation}
Then $\varphi$ is a strictly pseudoconvex function. 
\end{Theorem}
\begin{proof}  Observe that if $\nabla\varphi(x)=0$, then \eqref{SufficientPC} implies the positive semidefiniteness of $\partial^2\varphi(x)$  and so, by Proposition~\ref{SubfficientLocal}, $x$ is a strict local minimizer of $\varphi$. Hence, it follows from Theorem~\ref{QuasiPseudo} that $\varphi$ is strictly pseudoconvex if and only if $\varphi$ is quasiconvex. 
	
	Assume that $\varphi$ is not quasiconvex. Then, by Lemma~\ref{Quasiconvex}, there exist $x_1,x_2\in\mathbb{R}^n, x_1\ne x_2$ and $t_0\in(0,1)$ such that $\langle\nabla\varphi(x_1+t_0(x_2-x_1)),x_2-x_1\rangle=0$  and \eqref{LocalMax} is satisfied. Let $\bar{x}:=x_1+t_0(x_2-x_1)$ and $u:=x_2-x_1$. It follows that $u\ne 0$ and $\langle\nabla\varphi(\bar{x}),u\rangle=0$ and so, by \eqref{SufficientPC}, 
\begin{equation}\label{PositiveDefinite}
\langle z,u\rangle>0, \quad \forall z\in\partial^2\varphi(\bar{x})(u).
\end{equation}
For the sequence $x_k:=\bar{x}+(1/k)u$ ($k\in\mathbb{N}$) we have $x_k\rightarrow\bar{x}$. For sufficiently large k, we have $t_0+1/k\in(0,1)$ and so $\varphi(x_k)\leq\varphi(\bar{x})$ by \eqref{LocalMax}. Applying the classical mean value theorem, for sufficiently large $k$, there exists $\theta_k\in (0, 1/k)$ such that
\begin{equation}\label{term}
\langle\nabla\varphi(\bar{x}+\theta_ku),(1/k)u\rangle=\varphi(x_k)-\varphi(\bar{x})\leq 0.
\end{equation}
Consider the function $\phi:\mathbb{R}^n\rightarrow\mathbb{R}$ given by
$$
\phi(x)=\langle u,\nabla\varphi\rangle(x), \quad\forall x\in\mathbb{R}^n.
$$
 Applying Proposition~\ref{MeanValue}, for every $k$, there exist $\gamma_k\in(0,\theta_k]$ and $z_k\in\partial\phi(\bar{x}+\gamma_ku)$ such that
\begin{eqnarray*}
\langle z_k,-\theta_ku\rangle&\geq&\phi(\bar{x})-\phi(\bar{x}+\theta_ku)\\
&=&\langle u,\nabla\varphi(\bar{x})\rangle-\langle u,\nabla\varphi(\bar{x}+\theta_ku)\rangle\\
&=&-\langle u,\nabla\varphi(\bar{x}+\theta_ku)\rangle.
\end{eqnarray*}
Combining the above inequality with \eqref{term} we have $\langle z_k,u\rangle\leq 0$
for sufficiently large $k$. Since $\partial\phi$ is locally bounded at $\bar{x}$, the sequence $(z_k)$ is bounded. Without loss of generality, we can assume that $z_k\rightarrow z$. It follows that $\langle z,u\rangle\leq 0$ and by the robustness of $\partial\phi$ we have $z\in\partial\phi(\bar{x})=\partial^2\varphi(\bar{x})(u)$. This is contradict \eqref{PositiveDefinite}.
The proof is complete.
$\hfill\Box$
\end{proof}

\medskip
We consider two examples to analyze \eqref{SufficientPC}. The first one shows that \eqref{SufficientPC} cannot be relaxed to the following condition
\begin{equation}\label{Counter1}
x \in \mathbb{R}^n, u \in \mathbb{R}^n \setminus \{0\}, \langle \nabla \varphi(x), u \rangle=0 \Longrightarrow \langle z, u \rangle >0 \; \text{for some} \; z \in \partial^2 \varphi(x)(u)
\end{equation}
Moreover, \eqref{Counter1} is not sufficient for the quasiconvexity of $\varphi$.
\begin{Example}\rm
	Let $\varphi: \mathbb{R} \to \mathbb{R}$ be the function given by
	$$
	\varphi(x):= \int_0^{x} {\phi (t)dt},
	$$
	where
	$$
	\phi(t) = \begin{cases}
	-2t^2 + t^2 \sin(\frac{1}{t})& \text{if}\qquad t > 0,\\
	0 & \text{if}\qquad t= 0,\\
	2t^2 + t^2 \sin(\frac{1}{t}) & \text{if}\qquad t < 0.\\
	\end{cases}
	$$
\end{Example}
Observe that $\varphi$ is a $\mathcal{C}^{1,1}$-smooth function and $\nabla \varphi(x) = \phi(x)$ for every $x \in \mathbb{R}$.
Moreover,  we have $\partial^2 \varphi(0)(u) = \left[-|u|,|u|\right]$ for all $u \in \mathbb{R}$. Let $x \in \mathbb{R}$, $u \in \mathbb{R} \setminus \{0\}$ be such that $\langle \nabla \varphi(x), u \rangle = 0$. It follows that $\nabla \varphi(x) = 0$, or equivalently $x = 0$.  For   $z^* = u \in \partial^2 \varphi(0)(u) $, we have $\langle z^*, u \rangle = |u|^2 > 0$. The condition \eqref{Counter1} holds for $\varphi$. Howerver, $\varphi$ is not quasiconvex. Indeed, for $x =\displaystyle\frac{1}{\pi}, y =-\frac{1}{\pi} $, we have 
$$
\langle \nabla \varphi(x), y -x \rangle= \frac{4}{\pi^3} >0,\quad  \langle \nabla \varphi(y), y -x \rangle = -\frac{4}{\pi^3}<0.
$$
By \cite[Proposition~1]{Crouzeix98}, $\varphi$ is not quasiconvex.   

\medskip
The second example points out that we cannot replace the Mordukhovich second-order subdifferential in \eqref{SufficientPC} by the Fr\'echet second-order one since it may be empty.
\begin{Example} {\rm Let $\varphi: \mathbb{R} \to \mathbb{R}$ be the function given by 
		$$
		\varphi(x):= \int_0^{x} {\phi (t)dt} \quad \forall x \in \mathbb{R},
		$$
		where
		$$
		\phi(t) := \begin{cases}
		-2t - t\sin(\log(|t|))& \text{if}\qquad t \ne 0,\\
		0 & \text{if}\qquad t = 0
		\end{cases}
		$$
		is a Lipschitz continuous function.
		Hence, $\varphi$ is $\mathcal{C}^{1,1}$-smooth and $\nabla\varphi(x)=\phi(x)$ for every $x\in\mathbb{R}$. Let $x, u \in \mathbb{R}$, $u \ne 0$ such that $\langle\nabla \varphi(x),u\rangle=0$. Then, $\nabla\varphi(x) = 0$ and so $x = 0$. 
		We have $\widehat{\partial}^2 \varphi(0)(u) = \emptyset$. Thus, the below condition holds
		$$
		x \in \mathbb{R}, u \in \mathbb{R}\setminus \{0\}, \langle\nabla\varphi(x),u\rangle = 0 \Longrightarrow \langle z,u\rangle >0\; \text{for all} \; z \in \widehat{\partial}^2\varphi(x)(u). 
		$$
		However, $\varphi$ is not a pseudoconvex function. Indeed, for $x = 0, y = 1$, we have 
		$$
		\langle\nabla \varphi(x), y-x\rangle =0, \; \langle\nabla \varphi(y), y-x\rangle =-2 <0.
		$$
		By \cite[Proposition 2]{Crouzeix98}, $\varphi$ is not pseudoconvex. 
	}
\end{Example}

\medskip
When the Fr\'echet second-order subdifferential is nonempty, we can use it to characterize the strict quasiconvexity and strict pseudoconvexity of $\mathcal{C}^{1,1}$-smooth functions.

\begin{Theorem}\label{SufficientQC1}
	Let $\varphi:\mathbb{R}^n\rightarrow\mathbb{R}$ be a $\mathcal{C}^{1,1}$-smooth function satisfying
	\begin{equation} \label{SufQC}
	x \in \mathbb{R}^n, u \in \mathbb{R}^n \setminus \{0\}, \langle \nabla \varphi(x), u \rangle=0 \Longrightarrow  \langle z, u \rangle >0 \; \text{for some}\; z \in \widehat{\partial}^2\varphi(x)(u) \cup -\widehat{\partial}^2\varphi(x)(-u)
	\end{equation}
	Then $\varphi$ is a strictly quasiconvex function. 
\end{Theorem}
\begin{proof}
	Assume that $\varphi$ is not strictly quasiconvex. 	Then, by Lemma~\ref{Quasiconvex}, there exist $x_1,x_2\in\mathbb{R}^n$ with $x_1\ne x_2$ and $t_0\in(0,1)$ such that $\langle\nabla\varphi(x_1+t_0(x_2-x_1)),x_2-x_1\rangle=0$  and \eqref{LocalMax} is satisfied. Let $x:=x_1+t_0(x_2-x_1)$ and $u:=x_2-x_1$. It follows that $u\ne 0$ and $\langle\nabla\varphi(x),u\rangle=0$ and so, by \eqref{SufQC}, there exists $z\in \widehat{\partial}^2\varphi(x)(u) \cup -\widehat{\partial}^2\varphi(x)(-u)$ such that
	$\langle z,u\rangle>0$. Since 
	$$
	\widehat{\partial}^2\varphi(x)(u) \cup -\widehat{\partial}^2\varphi(x)(-u) = \widehat{\partial}\langle u,\nabla\varphi\rangle(x) \cup -\widehat{\partial}\langle -u,\nabla\varphi\rangle(x)
	$$ 
	it must happen one of the following cases.
	
	\noindent\medskip
	\textit{Case 1:} $z\in \widehat{\partial}\langle u,\nabla\varphi\rangle(x)$. 
	Since $\langle\nabla\varphi(x),u\rangle=0$, we have
\begin{equation}\label{split}
\begin{split}
0
&\leq \liminf_{y\rightarrow x}\frac{\langle u,\nabla\varphi(y)\rangle-\langle u,\nabla\varphi(x)\rangle-\langle z,y-x\rangle}{\|y-x\|}\\
&=\liminf_{y\rightarrow x}\frac{\langle u,\nabla\varphi(y)\rangle-\langle z,y-x\rangle}{\|y-x\|}.
\end{split}
\end{equation}	 
	For the sequence $x_k:=x+(1/k)u$ ($k\in\mathbb{N}$) we have $x_k\rightarrow x$. For sufficiently large k, we have $t_0+1/k\in(0,1)$ and so $\varphi(x_k)\leq \varphi(x)$ by \eqref{LocalMax}. Applying the classical mean value theorem, for sufficiently large $k$, there exists $\theta_k\in (0, 1/k)$ such that
	\begin{equation}\label{term1}
	\langle\nabla\varphi(x+\theta_ku),(1/k)u\rangle=\varphi(x_k)-\varphi(x)\leq 0.
	\end{equation}
	For the sequence $y_k:=x+\theta_ku\; (k\in\mathbb{N})$ we have $y_k\rightarrow x$ and $\langle\nabla\varphi(y_k),u\rangle\leq 0$ by \eqref{term1} for every $k\in\mathbb{N}$. It follows from \eqref{split} that
	\begin{equation*}\label{split2}
	\begin{split}
	0
	&\leq\liminf_{k\rightarrow\infty}\frac{\langle u,\nabla\varphi(y_k)\rangle-\langle z,y_k-x\rangle}{\|y_k-x\|}\\
	& \leq \liminf_{k\rightarrow\infty}\frac{-\langle z,\theta_ku\rangle}{\|\theta_ku\|}\\
	& =\frac{-\langle z,u\rangle}{\|u\|}
	\end{split}
	\end{equation*}
	which is contradict to $\langle z,u\rangle> 0$.	
	
	\noindent\medskip
	\textit{Case 2.} $z\in -\widehat{\partial}\langle -u,\nabla\varphi\rangle(x)$. 
	Repeating the proof of Case 1. with $u,z$ being replaced by $-u,-z$ we also get a contradiction. 
	$\hfill\Box$
\end{proof}
\begin{Remark}
	{\rm Observe that the strict quasiconvexity in Theorem~\ref{SufficientQC1} cannot be improved to strict pseudoconvexity.  Indeed, let $\varphi$ be the function given in Example \ref{Signfunction}.
		We have
		$$
		\widehat{\partial}^2\varphi(0)(u)\cup -\widehat{\partial}^2\varphi(0)(-u) =\left[-|u|,|u|\right],
		$$
		for all $u \in \mathbb{R}$. Observe that if $x \in \mathbb{R}$, $u\in \mathbb{R}\setminus\{0\}$ such that $\langle \nabla \varphi(x), u \rangle = 0$ then $x = 0$. Hence, with $z:= u \in \widehat{\partial}^2\varphi(0)(u)\cup -\widehat{\partial}^2\varphi(0)(-u)$ we have $\langle z, u \rangle = |u|^2 > 0$ and so \eqref{SufQC} holds  while $\varphi$ is not  strictly pseudoconvex. }
\end{Remark}

We now improve \eqref{SufQC} to get another characterization for the strict pseudoconvexity.
\begin{Theorem}
	Let $\varphi: \mathbb{R}^n \to \mathbb{R}$ be a $\mathcal{C}^{1,1}$-smooth function satisfying 
	\begin{equation}\label{Pseudoexists}
	x \in \mathbb{R}^n, u \in \mathbb{R}^n \setminus \{0\}, \langle \nabla \varphi(x), u\rangle =0 \Longrightarrow \langle z,u \rangle >0 \; \text{for some} \; z \in \widehat{\partial}^2\varphi(x)(u). 
	\end{equation}
	Then $\varphi$ is a strictly pseudoconvex function. 
\end{Theorem}
\begin{proof} By Theorem \ref{SufficientQC1}, $\varphi$ is strictly quasiconvex. 
	We will use Theorem~\ref{QuasiPseudo} to prove the strict pseudoconvexity of $\varphi$.
	Let $x \in \mathbb{R}^n$ such that $\nabla \varphi(x)=0$. It follows from \eqref{ScalarizationFormular} and \eqref{Pseudoexists}
	that 
	$$
	\widehat\partial\langle u,\nabla\varphi\rangle(x)=\widehat{\partial}^2 \varphi(x)(u)\ne \emptyset\quad\text{and}\quad 
	\widehat\partial\langle -u,\nabla\varphi\rangle(x)=\widehat{\partial}^2 \varphi(x)(-u)\ne \emptyset
	$$
	for every $u\in\mathbb{R}^n\setminus\{0\}$. 
	By \cite[Proposition~1.87]{Mordukhovich06}, the scalar function $\langle u, \nabla \varphi\rangle $ is differentiable at $x$  for every $u\in\mathbb{R}^n\setminus\{0\}$. Hence, $\varphi$ is twice differentiable at $x$ and 
	$$
	\widehat{\partial}^2\varphi(x)(u)=\{\nabla\langle u, \nabla \varphi\rangle (x)\}=\{\nabla^2\varphi(x)u\}
	$$
	 for every $u\in\mathbb{R}^n\setminus\{0\}$. Again, by \eqref{Pseudoexists}, its Hessian $\nabla^2\varphi(x)$ is positive definite. Moreover,
	by \cite[Theorem~13.2]{Rockafellar98}, the Hessian matrix  $\nabla^2\varphi(x)$ also furnishes a quadratic expansion for $\varphi$ at $x$.  Therefore, the positive definiteness of $\nabla^2\varphi(x)$ and the vanishing of $\nabla\varphi(x)$ yield that 
	 $x$ is a strict local minimizer of $\varphi$. By Theorem~\ref{QuasiPseudo}, $\varphi$ is strictly pseudoconvex. 
$\hfill\Box$
\end{proof} 
\begin{Remark}  {\rm The condition \eqref{Pseudoexists} implies that $\varphi$ is twice differentiable at every its critical point.}
\end{Remark}

\medskip
In the two next examples, we will show that \eqref{Pseudoexists} and \eqref{SufficientPC} are incomparable.
\begin{Example} {\rm 
		Let $\varphi: \mathbb{R} \to \mathbb{\mathbb{R}}$ be the function defined by
		$$
		\varphi(x) = \begin{cases}
		\frac{1}{2}x^2 & \text{if}\qquad x \leq 0,\\
		3x^2 & \text{if}\qquad x> 0.
		\end{cases}
		$$
		Then,   $\varphi$ is $\mathcal{C}^{1,1}$-smooth  and
		$$
		\nabla \varphi(x) = \begin{cases}
		x & \text{if}\qquad x \leq 0,\\
		6x & \text{if}\qquad x> 0.
		\end{cases}
		$$
		Let $x, u \in \mathbb{R}$, $u \ne 0$ such that $\langle \nabla \varphi(x),u \rangle =0$. Then, $\nabla \varphi(x) =0$ and so $x = 0$. Clearly,
		$$
		\widehat{\partial}^2 \varphi(0)(u) = \begin{cases}
		[u,6u] & \text{if}\qquad u \geq 0,\\
		\emptyset & \text{if}\qquad u< 0,
		\end{cases}
		\quad \text{and} \quad
		{\partial}^2 \varphi(0)(u) = \begin{cases}
		[u,6u] & \text{if}\qquad u \geq 0,\\
		\{u,6u\} & \text{if}\qquad u< 0.
		\end{cases}
		$$
		Hence,  \eqref{SufficientPC}  holds while \eqref{Pseudoexists} is not satisfied. 
	}
\end{Example}
\begin{Example} {\rm 
		Let $\varphi: \mathbb{R} \to \mathbb{R}$ be the function defined by
		$$
		\varphi(x):= \int_0^{x} {\phi (t)dt},
		$$
		where $\phi: \mathbb{R} \to \mathbb{R}$ is given by
		$$
		\phi(t) = \begin{cases}
		\displaystyle\frac{1}{2\pi} & \text{if} \qquad \displaystyle t\geq \frac{1}{\pi},\\
		\displaystyle\frac{t}{2} + t^2 \sin \left(\frac{1}{t} \right)& \text{if}\qquad \displaystyle 0<|t|<\frac{1}{\pi},\\
		0 & \text{if}\qquad t=0.\\
		\displaystyle-\frac{1}{2\pi} & \text{if} \qquad \displaystyle t\leq -\frac{1}{\pi}.
		\end{cases}
		$$
		Since $\phi$ is locally Lipschitz, $\varphi$ is  $\mathcal{C}^{1,1}$-smooth and $\nabla\varphi(x)=\phi(x)$ for every $x\in\mathbb{R}$. 
		Moreover, $\varphi$ is twice differentiable everywhere except the  points $\displaystyle \frac{1}{\pi}$ and $\displaystyle -\frac{1}{\pi}$.
		 Let $x, u \in \mathbb{R}^n$, $u \ne 0$ such that $\langle \nabla \varphi(x),u \rangle =0$. Then $\nabla \varphi(x)=0$, and so $x = 0$ since $\frac{1}{2} + x\sin(\frac{1}{x}) > 0$ for all $x \ne 0$.  Clearly,
		$$
		\widehat{\partial}^2 \varphi(0)(u) = \left\{ {\frac{1}{2}u} \right\}
		\quad \text{and} \quad
		{\partial}^2 \varphi(0)(u) =
		\begin{cases}
		\left[-\frac{1}{2}u,\frac{3}{2}u\right] &\text{if}\quad u\geq 0,\\
		\left[\frac{3}{2}u,-\frac{1}{2}u\right] &\text{if}\quad u<0.
		\end{cases}
		$$
		Hence,  \eqref{Pseudoexists} holds while \eqref{SufficientPC} is not satisfied. 
	}
\end{Example}

\section{Conclusions}
We have obtained several second-order necessary and sufficient conditions for the (strict) quasiconvexity and the (strict) pseudoconvexity of $\mathcal{C}^{1,1}$-smooth functions on finite-dimensional Euclidean spaces. Many examples are proposed to illustrate our results.
Further investigations are needed to generalize our results to wider classes of smooth and non-smooth functions on infinite-dimensional Banach spaces.

\end{document}